\documentclass[12pt,reqno]{amsart}

\usepackage{amsmath,amsthm,amsfonts,amscd,url,amssymb}
\usepackage{enumerate,hyperref}
\usepackage{tikz}
\usetikzlibrary{cd}
\usepackage{xcolor}
\usepackage{mathrsfs}

\usepackage[pagewise]{lineno}%\linenumbers
\usepackage{comment}
\usepackage{graphicx} % Required for inserting images
\input{xypic}
\xyoption{all}
\newcommand{\R}{\mathbb{R}}
\newcommand{\Z}{\mathbb{Z}}

\newcommand{\Nr}{N_{\R}}

\newcommand{\Mr}{M_{\R}}
\newcommand{\A}{\mathbb{A}}

\newcommand{\B}{\mathfrak{B}}

\newcommand{\calE}{\mathcal{E}}
\newcommand{\calS}{\mathcal{S}}

\newcommand{\mc}[1]{\mathcal{#1}}
\renewcommand{\t}[1]{\widetilde{#1}}

\newcommand{\rar}{\rightarrow}

\newcommand{\pb}[1]{{#1}^{*}}

\newcommand{\dirlim}[2]{\mathbf{#1}^{#2}}

\numberwithin{equation}{section}

\newtheorem{theorem}[equation]{Theorem}

\newtheorem{corollary}[equation]{Corollary}
\newtheorem{proposition}[equation]{Proposition}

\theoremstyle{definition}
\newtheorem{definition}[equation]{Definition}

\theoremstyle{remark}

\newtheorem{remark}[equation]{Remark}
\newtheorem{example}[equation]{Example}

\newtheorem{numpara}[equation]{}

\title{Equivariant sheaves on toric prevarieties}
\author{Jyoti Dasgupta}
\address{Department of Mathematics and Computing, Indian Institute of Technology (Indian School of Mines) Dhanbad, 826004, India}
\email{jdasgupta.maths@gmail.com }
\author{Kartik Roy}
\address{ Department of Mathematics, Ben Gurion University of the Negev, P.O.B. 653, Be'er Sheva 84105, Israel}
\email{kartik@post.bgu.ac.il}



\begin{document}

\begin{abstract}
Toric prevarieties are non-separated analogues of toric varieties. Perling \cite{Perling_equivariant_sheaves_tor_var} provided a combinatorial description of equivariant quasicoherent sheaves on toric varieties, extending earlier ideas of Klyachko, who outlined a general framework for equivariant torsion free sheaves in an unpublished work \cite{kly_sheaf}. In this article, we present a combinatorial description of equivariant quasicoherent sheaves on toric prevarieties. 
\end{abstract}
\maketitle
%\tableofcontents

\section{Introduction}

The theory of toric varieties offers a rich interplay between the geometric properties of algebraic varieties equipped with algebraic torus actions and the combinatorial structures underlying their construction. The study of equivariant vector bundles on toric varieties has been a focal point of research since the foundational work of Kaneyama, who characterised such bundles over nonsingular complete toric varieties \cite{kane, kaneyama2}. Later, Klyachko elegantly classified equivariant vector bundles over arbitrary toric varieties via families of decreasing filtrations on a fixed vector space, subject to compatibility conditions \cite{Klyachko_bundles_torvar}. In an unpublished work, Klyachko also sketched a generalisation of this framework to equivariant torsion free sheaves \cite{kly_sheaf}. Perling then formalised this using the notion of $\Delta$-families—structures derived from torus-eigenspace decompositions and multiplication maps—to establish an equivalence between such families and equivariant quasi-coherent sheaves \cite{Perling_equivariant_sheaves_tor_var}. For torsion free sheaves, these $\Delta$-families reduce to multifiltrations of subspaces of a fixed vector space, satisfying certain compatibility conditions. In the reflexive case, the data simplifies further to families of increasing filtration of a fixed vector space, without any compatibility condition.

Toric prevarieties are non-separated analogues of toric varieties. They occur as universal ambient spaces in algebraic geometry: every normal variety over an algebraically closed
field admits a closed embedding into a toric prevariety \cite{W}. The category of toric varieties is equivalent to the category of fans. Generalising this, A'Campo-Neuen and Hausen \cite{AH_toric_prevar} have introduced the concept of a system of fans and established an equivalence between the category of toric prevarieties and the category of the system of fans.

In this paper, we extend Perling’s combinatorial classification of equivariant sheaves from toric varieties to toric prevarieties. We provide a detailed description of equivariant locally free sheaves, examining them both through Klyachko’s filtration-theoretic framework and via the perspective of piecewise linear maps as developed by Kaveh and Manon \cite{KM22}. We also discuss the equivariant splitting of locally free sheaves.

\section{Preliminaries} 
In this section, we briefly review some basic definitions and results on toric prevarieties and torus equivariant sheaves on toric varieties, which we will use later.

\subsection{Toric prevarieties}

Let us begin by recalling the description of the category of toric prevarieties in terms of systems of fans that A’Campo-Neuen and  Hausen developed in \cite{AH_toric_prevar}.

Let  \(T \cong (k^*)^n\) be the \(n\)-dimensional algebraic torus, where \(k\) is any algebraically closed field of characteristic zero. 

\begin{definition}\cite[Definition 1.1, Definition 1.5]{AH_toric_prevar}\label{torprevardef}
A \emph{toric prevariety} is a normal integral scheme $X$ that is of finite type (but not necessarily separated) over $k$ together with a regular action of an algebraic torus $T$, and the choice of a point $x_0\in X$ such that $T\hookrightarrow X$ given by $t\mapsto t\cdot x_0$ is an open embedding. The point $x_0$ is referred to as the base point in the open orbit of $X$.

 A morphism $f\colon X\rightarrow X'$ of toric prevarieties $(X,x_0)$ and $(X',x_0')$ is said to be a \emph{toric morphism}, if $f(x_0)=x_0'$ and there is a homomorphism $\phi\colon T\rightarrow T'$ of algebraic tori such that $f(t\cdot x)=\phi(t)\cdot f(x)$ for all $t\in T$ and $x\in X$. 
\end{definition}

A toric prevariety is a (normal) toric variety in the sense of \cite{Fulton_toric_var} if and only if it is separated. The category of toric varieties is equivalent to that of fans. To provide a similar description for the category of toric prevarieties, this notion was generalised to introduce the notion of a system of fans, which we describe below(see \cite[Section 2]{AH_toric_prevar}).

Let \(M=\text{Hom}(T, k^*)  \cong \Z^n\) be the character lattice of \(T\) and \(N=\text{Hom}(M, \Z)\) be the dual lattice. We denote by \(\langle , \rangle : M \times N \rightarrow \Z \), the natural pairing between \(M\) and \(N\). Let $\Nr := N \otimes_{\Z} \R$ (resp. $\Mr := M \otimes_{\Z} \R)$ be the real vector space associated with $N$ (resp. $M$). The pairing $ M \times N \rightarrow \Z$ induces a pairing $\Mr \times \Nr \rightarrow \R$. 
We call a finite set $\Delta$ of strictly convex rational polyhedral cones in
$\Nr$ a {\it fan\/} in $\Nr$ if any two cones of $\Delta$ intersect in a
common face, and $\Delta$ contains all faces of all its cones. Let $\Delta$ be a fan in \(N_\R \) which defines a  toric variety \(X(\Delta)\) of dimension \(n\) over \( k\) under the action of \(T\). In order to fix some notations, we recall that $X(\Delta)$ is
constructed as follows:

Let $\sigma$ be a strongly convex rational polyhedral cone in $\Delta$. Its dual cone $\sigma^{\vee}$ lives in $M_\R$. The set of lattice points of $\sigma^{\vee}$ is \(\sigma_M:=\sigma^{\vee} \cap M \), an affine semigroup. Let \(U_{\sigma}:=\text{Spec }k[\sigma_M]\) be the affine toric variety corresponding to $\sigma$. The toric variety $X(\Delta)$ is constructed by gluing these affine toric varieties according to the gluing data encoded in the fan structure. A fan $\Delta$ is called {\it irreducible\/} if it consists of all the faces of a single cone $\sigma$, i.e., the associated toric variety is affine.

\begin{definition}\cite[Definition 2.1]{AH_toric_prevar} \label{def:sof}
Let $I$ be a finite indexing set. A collection $\calS=(\Delta_{ij})_{i,j\in I}$ of fans in $N_\R$ is said to be a \emph{system of fans}, if $\Delta_{ij}=\Delta_{ji}$ for all $i,j\in I$ and 
$\Delta_{ij}\cap\Delta_{jk}$ is a subfan of $\Delta_{ik}$ for all $i,j,k\in I$. 
\end{definition}

Observe that for all $i,j\in I$, the fan $\Delta_{ij}$ is a subfan of $\Delta_{ii}$.
Now, to a system of fans $\calS=(\Delta_{ij})_{i,j\in I}$, we may naturally associate a toric prevariety $X = X(\calS)$: consider the toric varieties $X_i :=X(\Delta_{ii})$ for $i\in I$ and the open subvarieties $X_{ij} :=X(\Delta_{ij})$ of $X_i$ (resp. $X_j$) for $i,j\in I$ with $i\neq j$. By definition, we have a toric isomorphism $f_{ji}\colon X_{ij}\xrightarrow{\sim} X_{ji}$ such that $f_{ki}=f_{kj}\circ f_{ji}$ for all $i,j,k\in I$. Therefore, we may glue the $X_i$'s using $f_{ij}$'s and obtain a toric prevariety $X(\calS)$.

\begin{example}
 Every toric variety is a toric prevariety where the indexing set of the system of fans is a singleton set.
\end{example}

\begin{example}\label{example_A1twoorigins}
Let $I=\{1,2\}$ and consider the one-dimensional cone $\sigma=\mathbb{R}_{\geq 0}$. We define a system of fans $\calS$ in $\mathbb{R}$ by setting $\Delta_{11}:=\{\sigma,\{0\}\}, \, \Delta_{22}:=\{\sigma,\{0\}\}$ and $\Delta_{12}:= \{\{0\}\}, \, \Delta_{21}:=\{\{0\}\}$. Then the corresponding toric prevariety $X(\calS)$ is given by taking two copies of $\A_k^1$, namely $X(\Delta_{11})$ and $X(\Delta_{22})$, and gluing them along $k^*\subseteq \A_k^1$ via the identity $k^*=X(\Delta_{12})=X(\Delta_{21})=k^*$. So $X(\calS)$ is the affine line with double origin.  
\end{example}

By Sumihiro's Theorem, every toric prevariety can be covered by finitely many torus-invariant open affine subspaces (see \cite[Proposition 1.3]{AH_toric_prevar}). In fact, every toric prevariety has two distinguished systems of charts: one is given by maximal torus-invariant separated open charts, the other by maximal torus-invariant affine open charts. Each element of the latter system of charts is an affine toric variety, which corresponds to an irreducible fan. Hence, we obtain the system of fans $\calS=(\Delta_{ij})_{i,j \in I}$ associated to the toric prevariety, whose diagonal fans $\Delta_{ii}$ are irreducible. We call this an \emph{affine} system of fans.

\begin{remark} \label{rem:Sfamilyproperties} Fix a system of fans $\calS$ indexed by $I$ and let $i, j ,k \in I$.
The following relations hold in $S$ and $X(\calS)$.
    \begin{enumerate}
        \item $\Delta_{ij}$ is a subfan of $\Delta_{ii}$ and $\Delta_{jj}$.
        \item One may not expect that $\Delta_{ij}$ and $\Delta_{ii} \cap \Delta_{jj}$ are
equal (see Example \ref{example_A1twoorigins}, Example \ref{ex:homog-spectra}).
        \item $X_{ij} = X(\Delta_{ij}) = X_i \cap X_j = X(\Delta_{ii}) \cap X(\Delta_{jj})$. 
        \item $\Delta_{ij} \cap \Delta_{jk} = \Delta_{jk} \cap \Delta_{ik} = \Delta_{ij} \cap \Delta_{ik}$ by Definition \ref{def:sof}, corresponding to the toric variety $X_{ijk}:= X_i \cap X_j \cap X_k$. 
        %(Simple proof of equalities:  $\Delta_{ij} \cap \Delta_{jk} = \Delta_{ij} \cap \Delta_{jk} \cap \Delta_{ik} = \Delta_{ij} \cap \Delta_{ik}$ Since $\Delta_{ij} \cap \Delta_{ik}$ is a subfan of $\Delta_{jk}$.)
    \end{enumerate}
We denote by 
$$\beta^{ijk}_{ij}: X_{ijk}  \longrightarrow X_{ij} \quad
\beta^{ij}_{i}: X_{ij} \longrightarrow X_i \quad \text{and } \beta^{ijk}_i := \beta^{ij}_{i} \circ \beta^{ijk}_{ij}$$
the equivariant open embeddings induced by the system of fans.
\end{remark}

Consider the set $\mathfrak{F}(\calS)
:= \{(\sigma,i); \; i \in I,\sigma \in \Delta_{ii}\}$ of labelled
cones associated to the system of fans $\calS$. We have a
natural equivalence relation on $\mathfrak{F}(\calS)$, given by
$$(\sigma,i) \sim (\tau,j) \quad \Longleftrightarrow \quad
\sigma =\tau \in\Delta_{ij}\,.$$
We call this equivalence relation the {\it
  glueing relation\/} of $\calS$, and we denote the set of equivalence
classes by $\Omega(\calS)$. The equivalence class of an element
$(\sigma, i) \in \mathfrak{F}(\calS)$ is denoted by $[\sigma,i] \in \Omega(\calS)$.
Furthermore, the face relation of cones induces a partial ordering on the set $\Omega(\calS)$, namely
$$[\tau, j ] \preceq [\sigma, i] \textrm{ if } \tau \textrm{ is a face of } \sigma \textrm{ and } [\tau, i]=[\tau, j].$$

The orbit-cone correspondence for toric varieties generalises to the set-up of toric prevarieties in the following fashion. For every $[\sigma,i] \in \Omega(\calS)$, there is a
{\it distinguished point} $x_{(\sigma,i)}$ in the corresponding $T$-orbit in $X_i$. In the toric prevariety $X(\calS)$ a point $x_{(\sigma,i)}$ is identified with $x_{(\tau,j)}$ if
and only $[\sigma,i]=[\tau,j]$. In fact, we have a bijection from $\Omega(\calS)$ to the set of $T$-orbits of the toric
  prevariety $X(\calS)$, given by $[\sigma,i] \mapsto T \cdot x_{[\sigma,i]}$.

The point $x_{0} := x_{[\{0\},i]}$ corresponding to the open $T$-orbit
will be considered as the {\it base point} of $X(\calS)$.  For a
distinguished point $x_{[\sigma,i]}$ of $X(\calS)$, we define
$X_{[\sigma,i]}$ to be the open affine $T$-invariant subspace of
$X(\calS)$ that contains $T \cdot x_{[\sigma,i]}$ as closed subset. A point $x_{[\sigma, i]}$ lies in the closure of the orbit $T 
 \cdot x_{[\tau,j]}$ if and only if $[\tau, j] \preceq [\sigma, i]$. In
  particular, one has
  $$
  X_{[\sigma,i]} = \bigcup_{[\tau, j] \preceq [\sigma, i]} T \cdot
  x_{[\tau,j]} = \bigcup_{[\tau, j] \preceq [\sigma, i]} X_{[\tau,j]}.
  $$
We have the following notion of morphism in the category of system of fans.
\begin{definition}\cite[Definition 3.1]{AH_toric_prevar}
   Let $\calS=(\Delta_{ij})_{i,j\in I}$ and $ \calS'=(\Delta'_{ij})_{i,j\in I'}$ be two system of fans in the vector spaces $N_\R$ and $N'_{\R}$, associated to the lattices $N$ and $N'$, respectively. A \emph{morphism of system of fans} $\calS \rar \calS'$ is a pair $(F, \mathfrak{f})$ such that $F: N \rar N'$ is a map of lattices, and $\mathfrak{f}: \Omega(\calS) \rar \Omega(\calS')$ is a map such that the following hold:
   \begin{enumerate}
       \item if $[\tau, j] \preceq [\sigma, i]$, then $\mathfrak{f}([\tau, j]) \preceq \mathfrak{f}([\sigma, i])$, i.e., $\mathfrak{f}$ is order-preserving.
       \item if $\mathfrak{f}([\sigma, i])=[\sigma', i']$, then $F_{\R}((\sigma)^{\circ}) \subset (\sigma')^{\circ}$, where $F_{\R} : N_{\R} \rar N'_{\R}$ is the linear map induced from $F$. Here $\sigma^\circ$ denotes the relative interior of the cone $\sigma$.
   \end{enumerate}
\end{definition}
\noindent
With this notion, we have the following equivalence:
\begin{theorem}\cite[Theorem 3.6]{AH_toric_prevar}
    The category of system of fans is equivalent to the category of toric prevarieties. The restriction of the functor to the full subcategory of affine system of fans also establishes an equivalence to the category of toric prevarieties.
\end{theorem}
 From now on, we assume that our system of fans is affine.

\subsection{Equivariant sheaves on toric varieties}

 In this subsection, we briefly recall Perling's \cite{Perling_equivariant_sheaves_tor_var} combinatorial description of equivariant sheaves on toric varieties.

\subsubsection{Quasicoherent sheaves}

Let $X\,=\,X(\Delta)$  be a toric variety corresponding to a fan $\Delta$ and $\mathcal{E}$ an equivariant quasicoherent sheaf on $X$. Let \(U_{\sigma}:=\text{Spec }k[\sigma_M]\) be the affine toric variety corresponding to a cone $\sigma \in \Delta$. The \(T\)-action on $\mathcal{ E }$ induces an action of \(T\) on the \(k[\sigma_M]\)-module \(E^{\sigma}:=\Gamma(U_{\sigma}, \mathcal{E})\) and consequently, the \(T\)-isotypical decomposition \(E^{\sigma}=\bigoplus\limits_{m \in M} E^{\sigma}_{m}\). 
%, which makes \(E^{\sigma}\) naturally an \(M\)-graded \(k[\sigma_M]\)-module as follows. 
%Recall that the action of the torus on \(U_{\sigma}\) induces the \(T\)-isotypical decomposition $k[\sigma_M]=\bigoplus\limits_{m \in \sigma_M} k \chi^m$. 
 The \(M\)-graded \(k[\sigma_M]\)-module structure on \(E^{\sigma}\) is given by the following multiplication: \(\chi^{\sigma}_{m,m'} : E^{\sigma}_{m} \rightarrow E^{\sigma}_{m'}, ~ e \mapsto \chi^{m'-m} \cdot e,\) where \(m, m' \in M\) and \(m'-m \in \sigma_M\).

Motivated by this, for each $\sigma \in \Delta$, Perling defined the directed preorder $\le_{\sigma}$ on the lattice $M$ as follows:
    for $m, m' \in M$, we say that $m \le_{\sigma} m'$ if and only if $m' - m \in \sigma_M$. We write $m < m'$ if $m \le_{\sigma} m'$, but not $m' \le_{\sigma} m$.
\begin{comment}

We have the following results:
\begin{enumerate}
    \item $\le_{\sigma}$ defines a directed preorder on M.
    \item $m \le_{\sigma} m'$ and $m' \le_{\sigma} m$ iff $m' - m \in \sigma^{\perp}$.
    \item If $\tau \preceq \sigma$, then $m \le_{\sigma} m'$ implies $m \le_{\tau} m'$.
    \item If $\sigma$ is of maximal dimension then $\le_{\sigma}$ is an order.
\end{enumerate}
\end{comment}
%The directed preorder $\le_{\sigma}$ on $M$ induces the following category of diagrams on the category of $k$-vector spaces. 
%Each diagram is called a $\sigma$-family. 

\begin{definition}\cite[Definition 5.2, Definition 5.4]{Perling_equivariant_sheaves_tor_var}\label{def:sigmaFamily} A $\sigma$\emph{-family}, denoted by $\hat{E}^{\sigma}$, is a family \{$E^{\sigma}_m\}_{m \in M}$ of $k$-vector spaces together with a linear homomorphism $\chi^{\sigma}_{m,m'} : E^{\sigma}_m \longrightarrow E^{\sigma}_{m'}$  whenever $m \le_{\sigma} m'$ in $M$, such that the following hold:
    \begin{enumerate}
        \item  $\chi^{\sigma}_{m,m} = 1$
        \item $\chi^{\sigma}_{m',m''}\circ\chi^{\sigma}_{m,m'} = \chi^{\sigma}_{m,m''}$ for any triple $m \le_{\sigma} m' \le_{\sigma} m''$ in $M$. 
    \end{enumerate}

   Let $\hat{E}^{\sigma}$ and $\hat{F}^{\sigma}$ be two $\sigma$-families with $k$-linear homomorphisms $\chi^{\sigma}_{m,m'}$ and $\psi^{\sigma}_{m,m'}$, respectively. A morphism from $\hat{E}^{\sigma}$ to $\hat{F}^{\sigma}$ is a collection of linear homomorphisms \{$\phi^{\sigma}_m : E^{\sigma}_m \longrightarrow F^{\sigma}_m\}_{m \in M}$ such that $$\phi^{\sigma}_{m'}\circ \chi^{\sigma}_{m,m'} = \psi^{\sigma}_{m,m'} \circ \phi^{\sigma}_m$$ whenever $m \le_{\sigma} m'$ in $M$.
\end{definition}

The category of $\sigma$-families is a $k$-linear abelian category. We have the following important result on the equivalence of categories.

   \begin{proposition}\cite[Proposition 5.5]{Perling_equivariant_sheaves_tor_var} \label{pro:sigmafamilyEquivalence}
     The following categories are equivalent:
     \begin{enumerate}
         \item The category of equivariant quasicoherent sheaves over $U_{\sigma}$,
         \item The category of $M$-graded $k[\sigma_M]$ modules with morphisms of degree 0,
         \item The category of $\sigma$-families.
     \end{enumerate}
   \end{proposition}

Let $\tau$ be a face of $\sigma$, denoted by $\tau \preceq \sigma$. Let \(i_{\tau \sigma} : U_{\tau} \hookrightarrow U_{\sigma}\) be the inclusion associated to the face relation. Let \(\widehat{E}^{\sigma}\) be a $\sigma$-family and \(E^{\sigma}:=\bigoplus_{m \in M} E^{\sigma}_m \) be the corresponding \(M\)-graded \(k[{\sigma}_M]\)-module. The pullback \(i_{\tau \sigma}^*E^{\sigma}=E^{\sigma} \otimes_{k[{\sigma}_M] } k[{\tau}_M]\) is naturally an \(M\)-graded \(k[{\tau}_M]\)-module (see \cite[Section 5.2]{Perling_equivariant_sheaves_tor_var}) and hence corresponds to a $\tau$-family by Proposition \ref{pro:sigmafamilyEquivalence}, which we denote by \(i_{\tau \sigma}^* \widehat{E}^{\sigma}\). Perling introduced the following notion of $\Delta$-families for a fan $\Delta$.

\begin{definition} \cite[Definition 5.8]{Perling_equivariant_sheaves_tor_var} \label{def:DeltaFamily} A collection $\{\widehat{E}^{\sigma} \}_{\sigma \in \Delta}$ of $\sigma$-families is called a $\Delta$\emph{-family}, denoted by $\widehat{E}^{\Delta}$, if for each pair $\tau \preceq \sigma$, there exists an isomorphism of families $\eta_{\tau \sigma}: i_{\tau \sigma}^{*}(\widehat{E}^{\sigma}) \cong \widehat{E}^{\tau} $ such that, for each triple $\rho \preceq \tau \preceq \sigma$, there is an equality $\eta_{\rho \sigma}=\eta_{\rho \tau} \circ  i_{\rho \tau}^*\eta_{\tau \sigma}$.\\
	A morphism of $\Delta$-families is a collection of morphisms $\{ \widehat{\phi}^{\sigma}: \widehat{E}^{\sigma} \rightarrow \widehat{F}^{\sigma}\}_{\sigma \in \Delta}$ such that for all $\sigma, \tau \in \Delta$ and $\tau \preceq \sigma$, the following diagram commutes:
	
	\begin{center}
		{\footnotesize
			\begin{tikzpicture}[description/.style={fill=white,inner sep=2pt}]
			
			\matrix (m) [matrix of math nodes, row sep=3em,
			column sep=3em, text height=1.5ex, text depth=0.25ex]
			{ i_{\tau \sigma}^{*}(\widehat{E}^{\sigma})  & &  i_{\tau \sigma}^{*}(\widehat{F}^{\sigma}) \\
				\widehat{E}^{\tau} & & \widehat{F}^{\tau}.\\ };
			%\draw[double,double distance=5pt] (m-1-1) – (m-1-3);
			\path[->] (m-1-1) edge node[auto] {$ i_{\tau \sigma}^{*} \widehat{\phi}^{\sigma}$}(m-1-3);
			\path[->] (m-2-1) edge node[below] {$\widehat{\phi}^{\tau}$}(m-2-3);
			\path[->] (m-1-1) edge node[auto] {$ \eta^E_{\tau \sigma}$}(m-2-1);
			\path[->] (m-1-3) edge node[auto] {$\eta^F_{\tau \sigma} $} (m-2-3);
			
			\end{tikzpicture}
		}
	\end{center}
	
\end{definition}

\begin{theorem}\label{the:TorEqvDelFamAndShvs} \cite[Theorem 5.9]{Perling_equivariant_sheaves_tor_var}
The category of $\Delta$-families is equivalent to the category of equivariant quasicoherent sheaves on the toric variety \(X(\Delta)\).
\end{theorem}

\begin{remark} \label{rem:pullbackfamily}
By the above theorem, for any toric morphism $f \colon X(\Delta') \rightarrow X(\Delta)$ of toric varieties, any equivariant quasicoherent sheaf $\calE$ on $X(\Delta)$ corresponds to a $\Delta$-family, say $\hat{E}^{\Delta}$, and the pullback sheaf $f^*(\calE)$ on $X(\Delta')$ corresponds to a $\Delta'$-family, which we denote by $f^*(\hat{E}^{\Delta})$.

\end{remark}

%\subsubsection{Coherent sheaves on toric varieties}
To describe coherent sheaves, we need to impose certain finiteness conditions on the families.
\begin{definition}\cite[Definition 5.10]{Perling_equivariant_sheaves_tor_var}\label{def:finSigmaFamily}
    A $\sigma$-family $\widehat{E}^{\sigma}$ is \emph{finite} if 
    \begin{enumerate}
        \item Each vector space $E^{\sigma}_{m}, m \in M$ is finite dimensional,
        \item for each chain $\cdots <_{\sigma} m_{i - 1 } <_{\sigma} m_i <_{\sigma} \cdots $ of characters in $M$ there exists $i_0 \in \Z$ such that $E^{\sigma}_i = 0 $ for all $i <_{\sigma} i_0$,
        \item there are only finitely many vector spaces $E^{\sigma}_m$ such the induced map 

        $$
        \bigoplus_{m' <_{\sigma} m} E^{\sigma}_{m'} \longrightarrow E^{\sigma}_m
        $$
        is not surjective.
    \end{enumerate}
    A $\Delta$-family $\{\widehat{E}^{\sigma}\}_{\sigma \in \Delta}$ is \emph{finite} if each $\sigma$-family $\widehat{E}^{\sigma}, \sigma \in \Delta$ is finite as in Definition \ref{def:finSigmaFamily}.
\end{definition}

\begin{proposition}\cite[Proposition 5.11]{Perling_equivariant_sheaves_tor_var}\label{pro:cohDeltaFamily}
    Let $\Delta$ be a fan and $X(\Delta)$ be the associated toric variety. Then a quasicoherent equivariant sheaf on $X(\Delta)$ is coherent if and only if the associated $\Delta$-family is finite.
\end{proposition}

\subsubsection{Torsion free sheaves on toric varieties}

Let $\mathcal{E}$ be an equivariant torsion free coherent sheaf on \(X(\Delta)\) and let $\widehat{E}^{\Delta}$ be its associated finite $\Delta$-family. For each $\sigma \in \Delta$, the $\sigma$-family $\widehat{E}^{\sigma}$ is a directed family of vector spaces. Let \(\mathbf{E}^{\sigma}:= \varinjlim\limits_{m \in M} E_m^{\sigma}\). In particular, for each $\sigma$-family $\widehat{E}^{\sigma}$ and relation $m \leq_{\sigma} m'$, we have the following commutative diagram

\begin{equation*}
    \begin{tikzcd}
        E^{\sigma}_m \arrow[rd,hook] \arrow[r, hook, "\chi^{\sigma}_{m,m'}"]  & E^{\sigma}_{m'} \arrow[d, hook] \\
        & \dirlim{E}{\sigma}
    \end{tikzcd}
\end{equation*}
Since $\mathcal{E}$ is torsion free, the maps in the above diagram are injective for all $\sigma \in \Delta$ (see \cite[Proposition 5.13]{Perling_equivariant_sheaves_tor_var}). Moreover, for a face $\tau \preceq \sigma$ the morphism $\widehat{E}^{\sigma} \hookrightarrow (i_{\tau \sigma})^{*} \widehat{E}^{\sigma}$ of $\sigma$-families is injective, as the restriction morphism $\Gamma (U_{\sigma}, \mathcal{E}) \hookrightarrow \Gamma (U_{\tau}, \mathcal{E})$ is injective.
Consequently the composition $\widehat{E}^{\sigma}\hookrightarrow i_{\tau \sigma}^* \widehat{E}^{\sigma} \stackrel{\sim}{\rightarrow} \widehat{E}^{\tau}
$
is injective (see Definition \ref{def:DeltaFamily}, \cite[Proposition 5.14]{Perling_equivariant_sheaves_tor_var}), which further induces a natural injection 
$\mathbf{E}^{\sigma}  \hookrightarrow \mathbf{E}^{\tau}$.
% \begin{proposition}\cite[Proposition 5.14]{Perling_equivariant_sheaves_tor_var}
%     Let $\mathcal{E}$ be a coherent sheaf on $X(\Delta)$ with its $\Delta$-family $\widehat{E}^{\Delta}$. If $\mathcal{E}$ is torsion-free then the homomorphisms $\widehat{E}^{\sigma} \rightarrow \widehat{E}^{\tau}$ are injections for any $\tau \preceq \sigma$ where $\tau, \sigma \in \Delta$.
% \end{proposition}
Now the injections $\widehat{E}^{\sigma} \rightarrow \widehat{E}^{\tau}$ induce injective maps $\dirlim{E}{\sigma} \rightarrow \dirlim{E}{\tau}$. Actually, these
inclusions are isomorphisms by the following corollary.

\begin{corollary}\cite[Corollary 5.16]{Perling_equivariant_sheaves_tor_var}
    The injections $\widehat{E}^{\sigma} \rightarrow \widehat{E}^{0}$ induce isomorphisms $\dirlim{E}{\sigma}  \cong \dirlim{E}{0} $.
\end{corollary}
Thus, all the vector spaces in the $\Delta$-family \(\widehat{E}^{\Delta}\) can be realised as vector subspaces of $\mathbf{E}^0$. This leads to the following definition of a family of multifiltrations. 

\begin{definition}\label{multfilt}\cite[Definition 5.17]{Perling_equivariant_sheaves_tor_var}
	Let $\Delta$ be a fan, \(V\) a finite-dimensional \(k\)-vector space, and let for each \(\sigma \in \Delta \) a set of vector subspaces \(\{E^{\sigma}_m\}_{m \in M } \) of \(V\) be given. We say that this system is a \emph{family of multifiltrations} of \(V\) if: 
	\begin{enumerate}
		\item[(i)] For each \(\sigma \in \Delta \) and \(m \leq_{\sigma } m' \), \(E^{\sigma}_m\) is contained in \(E^{\sigma}_{m'}\).
		\item[(ii)] \(V=\bigcup\limits_{m \in M} E^{\sigma}_{m} \) for each $\sigma \in \Delta$.
		\item[(iii)] For each chain \(\ldots <_{\sigma} { m_{i-1}} <_{\sigma} { m_i} <_{\sigma} \ldots \) of characters in \(M\) there exists an \(i_0 \in \Z \) such that \(E^{\sigma}_{m_i}=0\) for  \(i < i_0\).
		\item[(iv)] For every $\sigma \in \Delta$ there exist only finitely many vector spaces \(E^{\sigma}_m\) such that \(E^{\sigma}_m \nsubseteq  \sum\limits_{m' {<_{\sigma}} m} E^{\sigma}_{m'}\).
		\item[(v)] $($Compatibility condition.$)$ For each $\tau \preceq \sigma$ with \(S_{\tau} = S_{\sigma} + \Z_{\geq 0}(-m_{\tau} )\) we consider with respect to the
		preorder \(\leq_{\sigma }\) the ascending chains \(m + i \cdot m_{\tau}\) for \(i \geq 0\). By condition (iv) and because \(V\) is finite dimensional the sequence of subvector spaces \(E^{\sigma}_{m+ i \cdot m_{\tau}}\) necessarily becomes stationary for some \(i^{\tau}_m \in \Z \). We require that \(E^{\tau}_m= E^{\sigma}_{m+ i^{\tau}_m \cdot m_{\tau}}\) for all \(m \in M\).
	\end{enumerate}	

A \emph{morphism between families of multifiltrations} \(\{E^{\sigma}_m\}_{m \in M } \) and \(\{F^{\sigma}_m\}_{m \in M } \) is a homomorphism of the corresponding ambient vector spaces $\phi : \mathbf{E}^0 \rightarrow \mathbf{F}^0$ which is compatible with these multifiltratons, i.e.  $\phi(E^{\sigma}_m) \subseteq F^{\sigma}_m$.
\end{definition}

\begin{theorem}\cite[Theorem 5.18]{Perling_equivariant_sheaves_tor_var}\label{thm:Per torsion free}
    The category of equivariant torsion free coherent sheaves on the toric variety $X(\Delta)$ is equivalent to the category of families of multifiltrations of finite-dimensional vector spaces indexed by cones of $\Delta$.
\end{theorem}

\begin{remark}
    The above classification is related to the characterisation of torsion free sheaves
in \cite{kly_sheaf}. For details, see \cite[Section 5.5]{Perling_equivariant_sheaves_tor_var}. 
\end{remark}
\begin{remark} \label{rem:pullbackfiltration}
Let $\Delta'$ be a subfan of $\Delta$ such that the induced morphism $f:X(\Delta') \hookrightarrow X(\Delta)$ is an open embedding. By the above theorem, any equivariant torsion free coherent sheaf $\calE$ on $X(\Delta)$ corresponds to a family of multifiltration of a finite-dimensional vector space $\mathbf{E}^0$ indexed by cones of $\Delta$, say $\{E^{\sigma}_m\}_{m \in M}$. The restriction sheaf $\calE|_{X(\Delta')}$ on $X(\Delta')$ is again torsion free and corresponds to a family of multifiltrations of $\mathbf{E}^0$ indexed by cones of $\Delta'$, which we denote by $\{(f^*E)^{\sigma'}_m\}_{m \in M}$. 
\end{remark}

\subsubsection{Reflexive sheaves on toric varieties}

Let $\mathcal{ E }$ be an equivariant reflexive sheaf on \(X=X(\Delta)\). For each $\sigma \in \Delta$, we consider the open subset $\cup_{\rho \in \sigma(1)}U_{\rho}$ of $U_{\sigma}$, where $\sigma(1)$ is the collection of one-dimensional cones or rays of $\sigma$. By the orbit-cone correspondence, we know that the complement of this open subset has codimension at least $2$. Then, by reflexivity of $\mathcal {E}$, we have
$\Gamma(U_{\sigma}, \mathcal{E})_{m} \cong \bigcap_{\rho \in \sigma(1)} \Gamma( U_{\rho}, \mathcal{E})_{m} $
for each $m \in M$. In terms of multifiltrations of the limit vector space $\mathbf{E}^0$, this becomes $E^{\sigma}_{m} = \bigcap_{\rho \in \sigma(1)} E^{\rho}_{m}.$ Thus, an equivariant reflexive sheaf is completely determined by the family of multifiltrations \(\{E^{\rho}_m\}_{m \in M, \rho \in \Delta(1)}\) of the vector space $\mathbf{E}^0$ where $\Delta(1)$ is the collection of rays of the fan $\Delta$. Note that there is a canonical identification of \(M / ({\rho}^{\perp} \cap M)\) with \(\Z\) via the map \(m \mapsto \langle m, v_{\rho} \rangle \), where $v_{\rho}$ is the primitive ray generator of $\rho$. Hence, identifying  \(E^{\rho}_m=E^{\rho}(\langle m, v_{\rho} \rangle)\), we get an increasing full (integer) filtration of $\mathbf{E}^{0}$ given by
$ 0 \subseteq\hspace{4pt} \ldots \subseteq E^{\rho}(i) \subseteq E^{\rho}(i + 1) \subseteq \ldots \hspace{4pt}\subseteq \mathbf{E}^{0}.$ In fact, we have the following equivalence of categories:

\begin{theorem}\cite[Theorem 5.19]{Perling_equivariant_sheaves_tor_var}\label{the:reflexivshv}
    The category of equivariant reflexive sheaves on a toric variety $X$ is equivalent to the category of vector spaces with full filtration associated to each ray in its fan. The morphisms in this category are vector space homomorphisms which are compatible with filtrations in the $\Delta$-family sense.
\end{theorem}

\subsubsection{Locally free sheaves on toric varieties}

Note that locally free sheaves are reflexive and therefore given by filtrations of some vector space as in Theorem \ref{the:reflexivshv}.
The following equivalence of categories is due to Klyachko, which was also recovered by Perling's description.
\begin{theorem} \cite{Klyachko_bundles_torvar}\label{thm:klyachko}
    The category of vector bundles on a toric variety $X({\Delta})$ is equivalent to the category of vector spaces with family of increasing filtrations $(E,\{E^{\rho}(i)\}_{\rho \in \Delta(1), i \in \Z})$ which satisfy the following compatibility condition:

\noindent
    For each $\sigma \in \Delta$, there exists a $T$-eigenspace decomposition $E = \bigoplus_{m \in M/\sigma^\perp \cap M} E^{\sigma}_{m}$ such that
    $$E^{\rho}(i) = \sum_{m \in M/\sigma^\perp \cap M, \langle m, v_{\rho} \rangle \leq i} E^{\sigma}_{m}.$$
\end{theorem}

\section{Equivariant sheaves on toric prevarieties}

In this section, we generalise Perling's description to equivariant sheaves on toric prevarieties.

\subsection{Quasicoherent sheaves on toric prevarieties}
In this subsection, we give a combinatorial description of equivariant quasicoherent sheaves on toric prevarieties. The following definition is along the lines of Definition \ref{def:DeltaFamily}, and uses Remark \ref{rem:pullbackfamily}.
\begin{comment}

\textcolor{red}{
Let $X$ be a toric prevariety associated with a system of fans $\calS=(\Delta_{ij})_{i,j\in I}$.  For $i, j \in I$, we have the toric variety $X_{ij}=X(\Delta_{ij})$ and we denote $X_i := X_{ii}$. Moreover, $X_{ij} = X_{ji}$. We denote the equivariant open embeddings $f^{i}_{ij}: X_{ij} \hookrightarrow X_i$ and $f^{j}_{ij}: X_{ij} \hookrightarrow X_j$.
Let $\calE$ be an equivariant sheaf on $X$. Then, for each $i, j \in I$, we have the restricted equivariant sheaf $\calE \mid_{X_{ji}}$. Consider the pullback sheaf $f_{ij}^*(\calE \mid_{X_{ji}})$ on $X_{ij}$. By definition, there exists an isomorphism $$\eta^{\calS}_{ij} \colon f_{ij}^*(\calE \mid_{X_{ji}}) \xrightarrow{\sim} \calE \mid_{X_{ij}}.$$ Furthermore, we have $f_{ik}=f_{jk}\circ f_{ij}$ for all $i,j,k\in I$. Hence, we have the equality $$\eta^{\calS}_{ik}=\eta^{\calS}_{ij} \circ f_{ij}^*\eta^{\calS}_{jk}.$$ 
We define 
$$
\phi_{ij}:= f_{{*}_{ijk}}
$$
}

Remark \ref{rem:pullbackfamily} motivates us to define the following.
\end{comment}  

\begin{definition} \label{def:Sfam} Let $I$ be a finite set and $\calS = \{\Delta_{ij} : i,j \in I \}$ be an affine system of fans. A collection $\{\widehat{E}^{\Delta_{ii}} \}_{i \in I}$ of $\Delta$-families is called an $\calS$\emph{-family}, denoted by $\widehat{E}^{\calS}$, if for each $i,j \in I$, there exists an isomorphisms of $\Delta_{ij}$-families 

$$(\beta^{ij}_{i})^{*} (\widehat{E}^{\Delta_{ii}})  \xrightarrow[\cong]{\eta^{ij}} (\beta^{ij}_{j})^{*} (\widehat{E}^{\Delta_{jj}})$$
over $X_{ij}=X(\Delta_{ij})=X(\Delta_{ii}) \cap X(\Delta_{jj})$ and for each triple $i,j,k$ there is an equality 

$$\pb{(\beta^{ijk}_{ik})} ( \eta^{ik}) = \pb{(\beta^{ijk}_{jk})} ( \eta^{jk}) \circ  \pb{(\beta^{ijk}_{ij})} ( \eta^{ij})$$
of morphisms of $\Delta_{ij} \cap \Delta_{jk}$-families over $X_{ijk}$ (see Remark \ref{rem:Sfamilyproperties} and Remark \ref{rem:pullbackfamily}). Sometimes we say ``$\eta^{ik} = \eta^{jk} \circ \eta^{ij}$ over $X_{ijk}$'' for the last equality above.\\

A \emph{morphism of $\calS$-families} is a collection of morphisms of $\Delta$-families $\{ \widehat{\phi}^{\Delta_{ii}}: \widehat{E}^{\Delta_{ii}} \rightarrow \widehat{F}^{\Delta_{ii}}\}_{i \in I }$ 
such that for all $i,j \in I$, the following diagram commutes:
	
	\begin{center}
		{\footnotesize
			\begin{tikzpicture}[description/.style={fill=white,inner sep=2pt}]
			
			\matrix (m) [matrix of math nodes, row sep=3em,
			column sep=3em, text height=1.5ex, text depth=0.25ex]
			{ (\beta^{ij}_{i})^{*}(\widehat{E}^{\Delta_{ii}})  & &  (\beta^{ij}_{i})^{*}(\widehat{F}^{\Delta_{ii}}) \\
				(\beta^{ij}_{j})^{*}(\widehat{E}^{\Delta_{jj}})  & & (\beta^{ij}_{j})^{*}(\widehat{F}^{\Delta_{jj}}).\\ };
			%\draw[double,double distance=5pt] (m-1-1) – (m-1-3);
			\path[->] (m-1-1) edge node[auto] {$ (\beta^{ij}_{i})^{*} \widehat{\phi}^{\Delta_{ii}}$}(m-1-3);
			\path[->] (m-2-1) edge node[below] {$(\beta^{ij}_{j})^{*} \widehat{\phi}^{\Delta_{jj}}$}(m-2-3);
			\path[->] (m-1-1) edge node[auto] {$ \eta^{ij}_{E}$}(m-2-1);
			\path[->] (m-1-3) edge node[auto] {$\eta^{ij}_{F} $} (m-2-3);
			
			\end{tikzpicture}
		}
	\end{center}
	
\end{definition}
\noindent
We have the following generalisation of Theorem \ref{the:TorEqvDelFamAndShvs} to toric prevarieties. 

\begin{theorem}
Let $\calS$ be an affine system of fans and $X(\calS)$ the associated toric prevariety. The category of $\calS$-families is equivalent to the category of equivariant quasicoherent sheaves on the toric prevariety \(X(\calS)\).
\end{theorem}

\begin{proof} Let $\mathcal{E}$ be an equivariant quasicoherent sheaf over $X(\calS)$. Note that $\{X_i=X(\Delta_{ii}): i \in I\}$ is an invariant affine open cover of $X$. Furthermore, we have $X_i \cap X_j = X_{ij}$ and $X_i \cap X_j \cap X_k = X_{ijk}$ for all $i,j \in I$. For each $i \in I$, we consider the restriction $\mathcal{E}|_{X_i}$, which is an equivariant quasicoherent sheaf.  Then we obtain a family of equivariant quasicoherent sheaves $\mathcal{E}|_{X_i}$ over $X_i$ together with isomorphisms $\phi_{ij}:  \pb{(\beta^{ij}_{i})}\mathcal{E}|_{X_i} \longrightarrow \pb{(\beta^{ij}_{j})} \mathcal{E}|_{X_j}$ over $X_{ij}$ such that $\phi_{ii}= 1$  and 
 $\phi_{ik}= \phi_{jk} \circ \phi_{ij}$ over $X_{ijk}$. Now by Theorem \ref{the:TorEqvDelFamAndShvs}, each $\mathcal{E}|_{X_i}$ corresponds to a family of $\Delta_{ii}$-filtrations $E^{\Delta_{ii}}$ and for each $i,j \in I$, $\phi_{ij}$ correspond to a morphism $\eta^{ij}$ of $\Delta_{ij}$-families such that $\eta^{ii} = 1$ and $\eta^{ik} = \eta^{jk} \circ \eta^{ij}$ over $X_{ijk}$. Therefore, by Definition \ref{def:Sfam}, this corresponds to an $\calS$-family $\{ \widehat{E}^{\Delta_{ii}}\}_{i \in I}$.

 Conversely, let an $\calS$-family $\{ \widehat{E}^{\Delta_{ii}}\}_{i \in I}$ be given. Then for each $i \in I$, we have an equivariant quasicoherent sheaf $\mathcal{E}_i$ on the affine open $X_i$, by Theorem \ref{the:TorEqvDelFamAndShvs}. Moreover, the collection of equivariant quasicoherent sheaves $\mathcal{E}|_{X_i}$ over $X_i$ are glued over $X_{ij}$ using the isomorphism corresponding to the isomorphisms of $\Delta_{ij}$-families 

$$(\beta^{ij}_{i})^{*} (\widehat{E}^{\Delta_{ii}})  \xrightarrow[\cong]{\eta^{ij}} (\beta^{ij}_{j})^{*} (\widehat{E}^{\Delta_{jj}})$$
over $X_{ij}$, part of the data of the given $\calS$-family. Again by definition, these isomorphisms are compatible. Hence, we obtain an equivariant quasicoherent sheaf associated with the given $\calS$-family.
 \end{proof}

We now describe equivariant coherent sheaves on toric prevarieties combinatorially.

\begin{definition} \label{def:fin_Sfam}
We say that an $\calS$-family is \emph{finite} if all its $\Delta$-families are finite as in Definition \ref{def:finSigmaFamily}.
\end{definition}
\noindent
As an immediate consequence of Definition \ref{def:fin_Sfam}, Proposition \ref{pro:cohDeltaFamily}, and Theorem \ref{the:TorEqvDelFamAndShvs}, we have the following result.

\begin{proposition}
Let $\calS$ be a system of fans and $X(\calS)$ the associated toric prevariety. Then a quasicoherent equivariant sheaf on $X(\calS)$ is coherent if and only if the associated $\calS$-family is finite.
    
\end{proposition}

\subsection{Torsion free equivariant sheaves on toric prevarieties}

\begin{comment}

We describe a torsion free equivariant sheaf on a prevariety in terms of multifiltration developed in \textcolor{red}{Jose and Vivek}. 

\begin{definition}[Filtration] \label{def:filtration}
    Let $(P, \leq)$ be a partially ordered set (poset) and $V$ be a vector space (not necessarily finite dimensional). A decreasing (resp. increasing) filtration of $V$ with respect to $(P, \leq)$ is a family of subspaces $V_p, p \in P$ of $V$ such that for $p_1, p_2 \in P$ with relation $p_1 \leq p_2 $ the inclusion $V_{p_1} \geq V_{p_2}$ (resp. $V_{p_1} \leq V_{p_2}$) holds in $V$.
\end{definition}

\begin{example}
    The integers $\Z$ and the real numbers $\R$ are sets partially ordered with their natural ordering. In fact, these are ordered sets.
\end{example}

\begin{example}
    Let $\sigma$ be a cone in the lattice $N$. Then $\sigma$ induces a partial ordering $\leq_{\sigma}$ in the dual lattice $M$ as in Definition \textcolor{red}{def 1.5}. The necessary and sufficient condition for $\leq_{\sigma}$ to be an order is $\sigma$ is full dimensional in $N$.
\end{example}

\begin{definition} \label{def:multifiltration}
    Let $(P, \leq)$ and $V$ be as in Definition \ref{def:filtration}. Let $Q$ be a set. A multifiltration of $V$ with respect to $(P, \leq)$ indexed by $Q$ is a collection $\{V^{q}_{p}\}_{q \in Q; p \in P}$ of filtrations of $V$ with respect to $(P, \leq)$.
\end{definition}

\end{comment}

Following the footsteps of Perling, the notion of multifiltrations can be generalised to the following notion of $\calS$-multifiltrations, which we use to describe equivariant torsion free sheaves on toric prevarieties.

\begin{definition}\label{def:Smult} Let $I$ be a finite set and $\calS = \{\Delta_{ij} : i,j \in I \}$ be a system of fans and $E$ a finite-dimensional $k$-vector space. A collection $\{\{E^{\sigma}_m\}_{\sigma \in \Delta_{ii}, m \in M} \}_{i \in I}$ of family of $\Delta$-multifiltrations is called a \emph{family of $\calS$-multifiltrations}, if for each $i,j \in I$, there exists isomorphisms of families of $\Delta_{ij}$-multifiltrations 
\begin{equation}\label{torsion free equality 1}
   \{(\beta^{ij*}_{i} E)^{\sigma'}_m\}_{m \in M}  \xrightarrow[\cong]{\eta^{ij}} \{(\beta^{ij*}_{j} E)^{\sigma'}_m\}_{m \in M} 
\end{equation}
for all $\sigma' \in \Delta_{ij}$ and for each triple $i,j,k$ there is an equality 
\begin{equation}\label{torsion free equality}
    \pb{(\beta^{ijk}_{ik})} ( \eta^{ik}) = \pb{(\beta^{ijk}_{jk})} ( \eta^{jk}) \circ  \pb{(\beta^{ijk}_{ij})} ( \eta^{ij})
\end{equation}
of morphisms of family of $\Delta_{ij} \cap \Delta_{jk}$-multifiltrations over $X_{ijk}$ (see Remark \ref{rem:pullbackfamily} and Remark \ref{rem:Sfamilyproperties}).\\
A \emph{morphism of families of $\calS$-multifiltrations} is a homomorphism of vector spaces compatible with the multifiltrations.

\end{definition}

The following theorem is a generalisation of Theorem \ref{thm:Per torsion free}.
\begin{theorem} \label{the:torFreeCohSheaves}
    Let $X(\calS)$ be the toric prevariety associated with the system of fans $\calS$. The category of families of $\calS$-multifiltrations is equivalent to the category of equivariant torsion-free coherent sheaves of modules on the toric prevariety \(X(\calS)\).
\end{theorem}

\begin{proof}
    Let $\mathcal{E}$ be an equivariant torsion free sheaf on $X(\calS)$. For each $i \in I$, the restriction $\mathcal{E}|_{X_i}$ to the affine toric variety $X_i:=X(\Delta_{ii})$ is a torsion free sheaf and hence corresponds to a family of $\Delta$-multifiltrations indexed by cones of $\Delta_{ii}$, using Theorem \ref{thm:Per torsion free}. We can consider the restriction to $X_i \cap X_j= X(\Delta_{ij})$. Then the isomorphism of torsion free equivariant sheaves $$(\beta^{ij*}_{i} \mathcal{E})  \xrightarrow[\cong]{\eta^{ij}} (\beta^{ij*}_{j} \mathcal{E})$$on the toric variety $X_{ij}$ gives rise to the isomorphism of $\Delta_{ij}$-multifiltrations $$\{(\beta^{ij*}_{i} E)^{\sigma'}_m\}_{m \in M}  \xrightarrow[\cong]{\eta^{ij}} \{(\beta^{ij*}_{j} E)^{\sigma'}_m\}_{m \in M}.$$ Considering triple intersections, we have the toric variety $X_{ijk}$ corresponding to the fan $\Delta_{ij} \cap \Delta_{jk}$, and hence by definition, and using Theorem \ref{thm:Per torsion free}, we have the equality \eqref{torsion free equality}.

    Conversely, let $\{\{E^{\sigma}_m\}_{\sigma \in \Delta_{ii}, m \in M} \}_{i \in I}$ be a family of $\calS$-multifiltrations. Then consider the open cover of $X(\calS)$ by the maximal $T$-invariant affine charts $X_i:=X(\Delta_{ii})$. Associated to the family of $\Delta$-multifiltrations $\{E^{\sigma}_m\}_{\sigma \in \Delta_{ii}, m \in M}$, we have an equivariant torsion free sheaf $\mathcal{E}_i$ on $X_i$. The glueing data is given by the isomorphism of the sheaves $$(\beta^{ij*}_{i} \mathcal{E}_i)  \xrightarrow[\cong]{\eta^{ij}} (\beta^{ij*}_{j} \mathcal{E}_j)$$ given by \eqref{torsion free equality 1} and the compatibility condition associated with \eqref{torsion free equality}. Hence, we get an equivariant torsion free sheaf $\mathcal{E}$ on $X(\calS)$.
\end{proof}

\begin{comment}

\begin{example}
    Let $\calS$ be a system of fans in $N$. Let $\calS'$ be another system of fans constructed from $\calS$ by adding two or higher dimensional cones in fans in $\calS$. Then $X(\calS)$ sits in $X_{\calS'}$ as an open subscheme outside a subscheme of codimensional two or higher. Then from Theorem \ref{the:torFreeCohSheaves} the prevarieties $X(\calS)$ and $X_{\calS'}$ have equivalent categories of equivariant torsion-free coherent  sheaves of modules. 
\end{example}
\end{comment}
\subsection{Reflexive equivariant sheaves on toric prevarieties}
If $\mathcal{E}$ is a reflexive sheaf on a normal integral scheme $X$, then $\Gamma(X, \mathcal{E}) = \Gamma(X \setminus Y, \mathcal{E})$ where $Y$ is a closed subset of $X$ of
codimension at least two. For a toric prevariety $X(\calS)$, consider the closed subset $Y=\bigsqcup_{\{[\sigma, i] \in \Omega(\calS)| \dim \sigma \geq 2\}}T \cdot x_{[\sigma, i]}$. Hence, for a reflexive sheaf on $X(\calS)$, we have $\Gamma(X, \mathcal{E}) = \Gamma( \bigcup_{\{[\rho, i] \in \Omega(\calS)| \dim \rho=1\}}X_{\rho}, \mathcal{E})$, where $X_{\rho}$ is the affine toric variety corresponding to the one-dimensional cone $\rho$. Thus, we have the following theorem.

\begin{theorem}\label{thm:relexive}
    Let $X(\calS)$ be the toric prevariety associated with the system of fans $\calS$. The category of equivariant reflexive sheaves on the toric prevariety \(X(\calS)\) is equivalent to the category of full increasing filtrations of vector spaces $(E, \{E^{[\rho,i]}(s)\}_{[\rho, i] \in \Lambda, \, s \in \Z})$, where $\Lambda=\{[\rho,i] \in \Omega(\calS)| \dim (\rho)=1\}$.
\end{theorem}

\begin{proof}
    Let $\mathcal{E}$ be a reflexive sheaf on the toric prevariety $X(\calS)$. Then $\mathcal{E}$ is torsion free and by Theorem \ref{the:torFreeCohSheaves}, corresponds to a family of $\calS$-multifiltrations $\{\{E^{\sigma}_m\}_{\sigma \in \Delta_{ii}, m \in M} \}_{i \in I}$ of a vector space $\dirlim{E}{0}$. Here $\dirlim{E}{0}$ can be identified with the part of $\Gamma (T, \mathcal{E})$ of degree $0$. Then we can consider the subcollection $\{\{E^{\rho}_m\}_{\rho \in \Delta_{ii}(1), m \in M} \}$. Observe that, for $[\rho,i]=[\rho, j] \in \Lambda$, the corresponding subspaces are the same, and hence, we denote it by $E^{[\rho, i]}_m$. %Further, for each $i \in I$, $\Gamma (X_i,\mathcal{E})= \Gamma ( \cup_{\rho \in \Delta_{ii}(1)}X_{\rho},\mathcal{E})= \cap_{\rho \in \Delta_{ii}(1)}\Gamma (X_{\rho},\mathcal{E}).$ Recall that the fan $\Delta_{ii}$ is irreducible, i.e., consists of faces of a cone, say $\sigma \in \Delta_{ii}$. Thus, considering the $m$-th graded component, and translating to the family of multifiltrations, we obtain $E^{\sigma}_m=\cap_{\rho \in \Delta_{ii}(1)}E^{\rho}_m$.
    As in the case of toric varieties, for the primitive ray generator $v_{\rho}$ of the ray $\rho$, we have the identification $E^{[\rho, i]}_m= E^{[\rho, i]}(\langle m, v_{\rho} \rangle)$. This gives us increasing full filtrations of the vector space $\mathbf{E}^0$ indexed by $\Lambda$.

    Conversely, let us start with a collection of full increasing filtrations of a vector space $E$, given by $(E, \{E^{[\rho,i]}(s)\}_{[\rho, i] \in \Lambda, \, s \in \Z})$. Consider the open subset of $X(\calS)$ defined by $U=\bigcup_{\{[\rho, i] \in \Lambda\}}X_{\rho}$. 
For each $[\rho, i] \in \Lambda$, the collection of  filtrations $(E,\{E^{[\rho,i]}(s)\}_{  s \in \Z } )$ gives us an equivariant locally free sheaf $\mathcal{E}_\rho$ on the affine toric variety $X_\rho$ by Theorem \ref{thm:klyachko}. Since $X_{\rho} \cap X_{\rho'}=T$ for distinct rays $\rho$, the collection of equivariant locally free sheaves $\{\mathcal{E}_\rho\}_{[\rho, i] \in \Lambda}$ glue together to give an equivariant locally free sheaf $\mathcal{E}_U$ on $U$. Let us define $\mathcal{E}:=\alpha_*(\mathcal{E}|_U)$ to be the pushforward of $\mathcal{E}|_U$ via the open embedding $\alpha:U \hookrightarrow X(\calS)$.  Since the complement of $U$ has codimension at least $2$ in $X(\calS)$, we have that $\mathcal{E}$ is an equivariant reflexive sheaf on $X(\calS)$.
\end{proof}

\section{Equivariant locally free sheaves on toric prevarieties}
In this section, we focus on equivariant locally free sheaves, presenting two equivalent descriptions and analysing their equivariant splitting properties.
\subsection{Equivariant locally free sheaves on toric prevarieties as a collection of compatible filtrations} 
In this subsection, we provide a combinatorial description of locally free equivariant sheaves on toric prevarieties, generalising Theorem \ref{thm:klyachko}.
\begin{theorem}\label{th:klyachko-prevar}
    Let $X(\calS)$ be the toric prevariety associated with the system of fans $\calS$. The category of equivariant locally free sheaves on the toric prevariety \(X(\calS)\) is equivalent to the category of full increasing filtrations of vector spaces $(E, \{E^{[\rho,i]}(s)\}_{[\rho, i] \in \Lambda, \, s \in \Z})$ indexed by $\Lambda=\{[\rho,i] \in \Omega(\calS)| \dim(\rho)=1\}$ satisfying the following compatibility condition:\\
    for each $\sigma \in \Delta_{ii}$, there exists a $T$-eigenspace decomposition $E= \bigoplus_{m \in M/\sigma^\perp \cap M} E^{\sigma}_m$ such that $$E^{[\rho,i]}(s)=\sum_{\substack{m \in M/\sigma^\perp \cap M, \\ \langle m, v_{\rho} \rangle \leq s}} E^{\sigma}_m$$ for each $\rho \in \sigma(1).$
\end{theorem}
\begin{proof}
     Let $\mathcal{E}$ be an equivariant locally free sheaf on the toric prevariety $X(\calS)$. Then, since $\mathcal{E}$ is reflexive, we have associated a family of full decreasing filtrations of vector space $E$ indexed by $\{[\rho,i] \in \Omega(\calS)| \text{dim }(\rho)=1\}$. Since over any affine toric variety, an equivariant locally free sheaf is free, for each $i \in I$, from the restricted sheaf $\mathcal{E}|_{X_i}$, we have the $T$-eigenspace decomposition $E= \bigoplus_{m \in M /\sigma^\perp \cap M} E^{\sigma}_m$ such that $$E^{[\rho,i]}(s)=\sum_{\substack{m \in M/\sigma^\perp \cap M, \\ \langle m, v_{\rho} \rangle \leq s}} E^{\sigma}_m$$ for each $\rho \in \sigma(1)$. 

     Conversely, let us begin with a collection of full increasing filtrations of a vector space $E$. Then we have a reflexive sheaf $\mathcal{E}$ on $X(\calS)$ by Theorem \ref{thm:relexive}. Considering the open cover $\{X_i : i \in I\}$ of $X(\calS)$ the restricted sheaf $\mathcal{E}|_{X_i}, i \in I$ is locally free by Theorem \ref{thm:klyachko}. 
\end{proof}

\begin{remark}\label{rem:dist-lattice}
    Let $X(\calS)$ be a smooth toric prevariety. This is equivalent to saying that all cones $\sigma \in \Delta_{ii}$ are smooth. Then the compatibility condition can be restated as follows:  Let $(E, \{E^{[\rho,i]}(s)\}_{[\rho, i] \in \Lambda, \, s \in \Z})$ be a collection of full increasing filtrations of the vector space $E$. Then $(E, \{E^{[\rho,i]}(s)\}_{[\rho, i] \in \Lambda, \, s \in \Z})$ satisfies the compatibility condition if and only if, for each $\sigma \in \Delta_{ii}$, there exists a basis $B_{\sigma}$ of $E$ such that $B_{\sigma} \cap E^{[\rho,i]}(s)$ is a basis for $E^{[\rho,i]}(s)$ for all $\rho \in \sigma(1), \, s \in \Z$. In other words, $(E, \{E^{[\rho,i]}(s)\}_{[\rho, i] \in \Lambda, \, s \in \Z})$ forms a distributive lattice. The proof is similar to the case of toric varieties (see \cite[Remark 2.2.2]{Klyachko_bundles_torvar}). 
\end{remark}

The tangent bundle on a smooth toric prevariety has a combinatorial description analogous to the tangent bundle on a smooth toric variety.
\begin{proposition}\label{prop:tangent}     
Let $X(\calS)$ be a smooth toric prevariety. Then its tangent bundle $\mathscr{T}_X$ is given by the following filtrations of the vector space $N\otimes k$: 	\[ \mathscr{T}^{[\rho,i]}(s) = \left\{ \begin{array}{cc}
	%{r@{\quad \quad}l}
	0 & s \leq -2 \\ 
	\mathrm{Span}(v_{\rho}) & s=-1 \\
	N \otimes_{\Z} k & s \geq 0,
	\end{array} \right. \] 
where $[\rho,i] \in \Lambda$ and $v_{\rho}$ is the primitive ray generator of $\rho$.
\end{proposition}
\begin{proof}
Consider the open cover $\{X_i : i \in I\}$ by affine toric varieties of $X(\calS)$. For each $i \in I$, the restricted tangent bundle $\mathscr{T}_X|_{X_i}$ coincides with the tangent bundle of the affine toric variety $X_i$. Then the description follows from \cite[Section 2.3, Example 5]{Klyachko_bundles_torvar} (see also \cite[Corollary 2.2.17]{DDK}).
\end{proof}

\subsubsection{Equivariant splitting of locally free equivariant sheaves on toric prevarieties} 
We say that an equivariant locally free sheaf splits equivariantly if it is equivariantly isomorphic to a direct sum of line bundles. Then, Remark \ref{rem:dist-lattice} can be used to prove the following criterion for equivariant splitting.
\begin{proposition} \label{prop:splitting}
    Let $X(\calS)$ be a smooth toric prevariety. The equivariant locally free sheaf associated to $(E, \{E^{[\rho,i]}(s)\}_{[\rho, i] \in \Lambda, \, s \in \Z})$ splits equivariantly if and only if there exists a basis $B$ of $E$ such that $B\cap E^{[\rho,i]}(s)$ is a basis for $E^{[\rho,i]}(s)$ for all $[\rho, i] \in \Lambda, \, s \in \Z$.
\end{proposition}

The following result about distributive lattices will be useful in proving the structure results for toric prevarieties.
\begin{theorem}\cite[Theorem 6.1.2]{Klyachko_bundles_torvar}\label{thm:subfamily}
    A family of subspaces $\{E_a\}_{a \in A}$, of an $m$-dimensional vector space $E$ generates a distributive lattice if and only if each $(m+1)$-element subfamily generates a distributive lattice.
\end{theorem}

Depending on the structure of a toric prevariety, we can say when a vector bundle of a certain rank will always split equivariantly. Recall the definition of $\Lambda$ from Theorem \ref{th:klyachko-prevar}. The proof of the following corollary is along the lines of \cite[Corollary 6.1.4]{Klyachko_bundles_torvar}.
\begin{corollary}\label{cor:ranksplit}
    Let $X(\calS)$ be a smooth toric prevariety. The following are equivalent:
    \begin{enumerate}
        \item Any rank $m$ locally free equivariant sheaf on $X(\calS)$ splits equivariantly into direct sum of line bundles.

        \item For any subset $A$ of $\Lambda$ of cardinality $m+1$, there is a cone $\sigma_A \in \Delta_{jj}$ for some $j \in I$ such that $\sigma_A$ is generated by $\{\rho|[\rho, i] \in A\}$. 
    \end{enumerate}
\end{corollary}

\begin{proof}
    Let us assume that any rank $m$ locally free equivariant sheaf on $X(\calS)$ splits. Assume by contradiction that $A$ is a subset of $\Lambda$ of cardinality $m+1$ with the condition that there is no $j \in I$, such that $\{\rho|[\rho, i] \in A\}$ generates a cone of $\Delta_{jj}$. Let $E$ be an $m$-dimensional vector space over $k$ and $\{L_{[\rho,i]}|[\rho, i] \in A\}$ be $m+1$ lines in $E$ in general position. Then the collection of filtrations of vector spaces $(E, \{E^{[\rho,i]}(s)\}_{[\rho, i] \in \Lambda, \, s \in \Z})$ defined below satisfy the compatibility condition of Remark \ref{rem:dist-lattice}, and hence defines an equivariant locally free sheaf $\mc{E}$ of rank $m$ on the toric prevariety $X(\calS)$. \\
    $ E^{[\rho,i]}(s) = \left\{ \begin{array}{cc}
	%{r@{\quad \quad}l}
	0 & s \leq -2 \\ 
	L_{[\rho,i]} & s=-1 \\
	E & s \geq 0,
	\end{array} \right.$, for $[\rho, i] \in A$,
 $E^{[\rho,j]}(s) = \left\{ \begin{array}{cc}
	%{r@{\quad \quad}l}
	0 & s \leq -1 \\ 
	E & s \geq 0,
	\end{array} \right.$, for $[\rho, j] \notin A$.\\However, from Proposition \ref{prop:splitting}, it follows that $\mc{E}$ does not split equivariantly, which is a contradiction.

    Now, let us assume that (2) holds. Let $\mc{E}$ be an equivariant locally free sheaf of rank $m$ on $X(\calS)$, corresponding to the collection of filtrations of vector spaces $(E, \{E^{[\rho,i]}(s)\}_{[\rho, i] \in \Lambda, \, s \in \Z})$. Then by Remark \ref{rem:dist-lattice}, any $(m+1)$-family of filtrations forms a distributive lattice, and hence using the Theorem \ref{thm:subfamily} and Proposition \ref{prop:splitting}, we see that $\mc{E}$ splits equivariantly. 
\end{proof}

\begin{example}
Let $I=\{1,2\}$ and consider the three-dimensional cone $\sigma=\mathbb{R}_{\geq 0}(1,0,0) \oplus \R_{\geq 0} (0,1,0) \oplus \R_{\geq 0} (0,0,1)$ in $\R^3$. For $i \in \{1,2\}$, let $\Delta_{ii}$ be the fan of $\sigma$ and its faces. Let $\Delta_{12} = \{\rho_1=\mathbb{R}_{\geq 0}(1,0,0), \, \rho_2=\R_{\geq 0} (0,1,0), \, \rho_3=\R_{\geq 0} (0,0,1), \{0\}\}= \Delta_{21}$. Let $\calS$ be the system of fans given by $\{\Delta_{11}, \Delta_{22},\Delta_{12},\Delta_{21}\} $. Then $X(\calS)$ is $\mathbb{A}^3$ with doubled coordinate axes. It is obtained by patching two copies of $\A^3$ along the complement of the coordinate axes. It is a non-separated smooth toric prevariety. %$\{(0,y,z)| y,z \in k\}\cup \{(x,0,z)| x,z \in k\} \cup \{(x,y,0)| x,y \in k\}$. 
Observe that $\Lambda = \{[\rho_1,1]=[\rho_1,2], \, [\rho_2, 1]=[\rho_2,2], \, [\rho_3, 1]=[\rho_3,2] \}.$  There is only one possibility for a subset of $\Lambda$ of cardinality 3, which is $\Lambda$ itself. Then we can take $\sigma=\mathbb{R}_{\geq 0}(1,0,0) \oplus \R_{\geq 0} (0,1,0) \oplus \R_{\geq 0} (0,0,1)$ in $\Delta_{11}$ satisfying the second condition of Corollary \ref{cor:ranksplit}. Hence, any rank $2$ equivariant locally free sheaf on $X(\calS)$ splits.

\end{example}

\subsection{Equivariant locally free sheaves on toric prevarieties as piecewise linear maps} 
In \cite{KM22}, the authors classify equivariant locally free sheaves on toric varieties in terms of piecewise linear maps to the extended Tits building associated with the general linear group. In this subsection, we obtain a similar description of equivariant locally free sheaves on toric prevarieties.

\begin{definition}
       Let $E$ be a finite-dimensional vector space over $k$. We call a function $v :E \to \overline{\R}=\R \cup \{\infty\}$ a \emph{prevaluation} if the following hold:
\begin{enumerate}
    \item For all $e \in E$ and $c \in k \setminus \{0\}$ we have $v(ce) = v(e)$.
    \item (Non-Archimedean property) For all $e_1, e_2 \in E \setminus \{0\}$, $e_1 + e_2 \neq 0$, the non-Archimedean
inequality $v(e_1 + e_2) \geq \min\{v(e_1), v(e_2)\}$ holds.

\item $v(e)= \infty$ if and only if $e=0$.
\end{enumerate}
We say that a prevaluation is \emph{integral} if the image lies in $\overline{\Z}=\Z \cup \{\infty\}$.
\end{definition}

 \begin{definition} 
 Let $F:E \rar E'$ be a linear map of $k$-vector spaces. Given a prevalutaion $w :E' \to \overline{\R}$, the \emph{pullback} $F^*(w):E \to \overline{\R}$ is defined to be $F^*(w)=w \circ F$. The pullback is a prevaluation only if $F$ is injective.
 \end{definition}
 An integral prevaluation $v :E  \to \overline{\Z}$ defines an increasing full filtration $E_{\bullet}=(E_{v \geq -a})_{a \in \Z}$ of $E$ by $$E_{v \geq -a}:=\{e \in E \, | \, v(e) \geq -a\} .$$
Conversely, an increasing full filtration $E_{\bullet}=(E_a)_{a \in \Z}$ of a vector space $E$ defines an integral prevaluation $v :E  \to \overline{\Z}$ by \[v(e)=\min \{a \in \Z | e \in E_a\}.\]

We denote the set of prevaluations $v :E  \to \overline{\R}$ by $\t{\B}(E)$ and call it the \emph{extended Tits building of $E$}. Let $B=\{b_1, \ldots, b_r\}$ a basis of $E$. We denote the set of all prevaluations which satisfy $v(\sum \lambda_i b_i)= \min \{v(b_i)| \lambda_i \neq 0\}$ by $\t{A}(B)$ and call it the \emph{extended apartment of $B$}.

The following definition appears in \cite[pp 69--70]{KKMD}.
\begin{definition}
    Let $I$ be a finite indexing set. Let $\calS=(\Delta_{ij})_{i,j \in I}$ be a system of fans in $N_\R$. Define the \emph{topological space $|\calS|$ associated to $\calS$} as \[|\calS|=(\bigsqcup_{(\sigma, i) \in \mathfrak{F}(\calS)}|\sigma|)/\sim\] where $x \in |\sigma|$ for $(\sigma, i) \in \mathfrak{F}(\calS)$ is equivalent to $y \in |\tau|$ for $(\tau, j) \in \mathfrak{F}(\calS)$ if $[\sigma,i]=[\tau,j]$ in $\Omega(\calS)$ (in other words, $\sigma= \tau \in \Delta_{ij}$) and $x=y \in |\sigma|=|\tau|$. Here $|\sigma|$ means the support of the cone $\sigma$ in $N_{\R}$. 
\end{definition}

In the following $E, E'$
 denote finite-dimensional vector spaces over $k$.
\begin{definition}
We say that a map $\Phi: |\calS| \to \t{\B}(E)$ is a \emph{piecewise linear map} if the following conditions hold:
\begin{itemize}
\item[(a)] For each cone $[\sigma, i] \in \Omega(\calS)$, there is a basis $B_\sigma \subset E$ (not necessarily unique) such that $\Phi(|\sigma|)$ lies in the extended apartment $\t{A}(B_\sigma)$.
\item[(b)] ${\Phi}|_{|\sigma|}$ is the restriction of a linear map from the linear span of $\sigma$ to $\t{A}(B_\sigma)$. \end{itemize}
We say that $\Phi$ is an \emph{integral} piecewise linear map if for each $[\sigma,i] \in \Omega(\calS)$, ${\Phi}|_{|\sigma|}$ is the restriction of an integral linear map from $N_\R$ to $\t{A}(B_\sigma)$. Here, integral means it sends $N$ to integral prevaluations. 

Now, given two piecewise linear maps $\Phi: |\calS| \to \t{\B}(E)$ and $\Phi': |\calS| \to \t{\B}(E')$, a linear map $F:E \to E'$ is said to be a \emph{morphism} $F: \Phi \to \Phi'$ if for any $x \in |\calS|$ and any $e \in E$, we have $F^*(\Phi'(x))(e):=\Phi'(x)(F(e))$.
\end{definition}

Thus, we have the category of piecewise linear maps on $|\calS|$ and the subcategory of integral piecewise linear maps on $|\calS|$. With this framework, we can restate Theorem \ref{th:klyachko-prevar}.

\begin{theorem}
    The category of equivariant locally free sheaves on $X(\calS)$ is naturally equivalent to the category of integral piecewise linear maps from $|\calS|$ to $\widetilde{\mathfrak{B}}(E)$, for all finite-dimensional $k$-vector spaces $E$.
\end{theorem}

\begin{proof}
   Let $\mathcal{E}$ be an equivariant locally free sheaf on $X(\calS)$. We have the open cover of $X(\calS)$ by affine toric varieties $\{X_i:i \in I\}$. Then associated to $\mathcal{E}|_{X_i}$, we have an integral piecewise linear map $\Phi_i:|\Delta_i| \rar \widetilde{\mathfrak{B}}(E)$ by \cite[Theorem 2.2, Example 2.8]{KM22}. These $\{\Phi_i:i \in I\}$ patch together to give an integral piecewise linear map $\Phi:|\calS| \rar \widetilde{\mathfrak{B}}(E)$.

   The converse direction is similar and left to the reader.
\end{proof}

\section{Examples}
In this section, we consider a few specific smooth toric prevarieties and write the combinatorial description of their tangent bundles using Proposition \ref{prop:tangent}.
\begin{example}
    Consider the system of fans describing the affine line with double origin as in Example \ref{example_A1twoorigins}. Then the associated combinatorial data of the tangent bundle is:	\[ \mathscr{T}^{[\sigma,i]}(s) = \left\{ \begin{array}{cc}
	%{r@{\quad \quad}l}
	0 & s \leq -2 \\ 
	k & s\geq-1 .
	\end{array} \right. \] where $\sigma$ is the one-dimensional cone $\R_{\geq 0}$, and $i=1, \,2$.
\end{example}
\begin{example}
Let $\sigma_1 = \R_{\geq 0}(0,1) \oplus \R_{\geq 0}(1,-1)$ and $\sigma_2 = \R_{\geq 0} (1,0)\oplus \R_{\geq 0} (-1,-1 )$ in $\R^2$.  For $i \in \{1,2\}$, let $\Delta_{ii}$ be the fan of $\sigma_i$ and its faces. Let $\Delta_{12} = \{0\}$. Let $\calS$ be the system of fans consisting of affine fans $\Delta_{11}, \Delta_{22} $ and $\Delta_{12}$. Then $X(\calS)$ is a (non-separated) smooth toric prevariety (see Figure \ref{fig:toric-prevar}. 
Let $\rho_1 = \R_{\geq 0}(0,1), \rho_2 = \R_{\geq 0}(1,-1), \rho_3 = \R_{\geq 0}(1,0) \text{ and }\rho_4 =\R_{\geq 0} (-1,-1).$
Then, the indexing set is given by $\Lambda = \{[\rho_1,1], \, [\rho_2, 
 1], \, [\rho_3, 2], \, [\rho_4, 2] \}.$

 \begin{figure}[ht]
     \centering
      \begin{tikzpicture}[scale=0.5]
 \fill[blue!40!white, opacity=0.2] (0,0) -- (7,0) arc (0: -135:7) -- cycle;

  \fill[red!40!white, opacity=0.2] (0,0) -- (0,7) arc (90: -45:7) -- cycle;
     \draw[blue, thick, ->] (0,0) -- (7.5,0);

     \draw[red, thick, ->] (0,0) -- (0,7.5);

     \draw[red, thick, ->] (0,0) -- (5.5,-5.5);

     \draw[blue, thick, ->] (0,0) -- (-5.5, -5.5);
    \filldraw[cyan] (0,0) circle (0.2cm);
     \node at (3,3) {\textcolor{red}{$\sigma_1$}};

 \node at (0,7.8) {\textcolor{red}{$\rho_1$}};

  \node at (5.8,-5.8) {\textcolor{red}{$\rho_2$}};
          \node at (0,-3) {\textcolor{blue}{$\sigma_2$}};

           \node at (7.9,0) {\textcolor{blue}{$\rho_3$}};
            \node at (-5.8,-5.8) {\textcolor{blue}{$\rho_4$}};
 \end{tikzpicture}
     \caption{Maximal invariant affine open chart of the toric prevariety}
     \label{fig:toric-prevar}
 \end{figure}
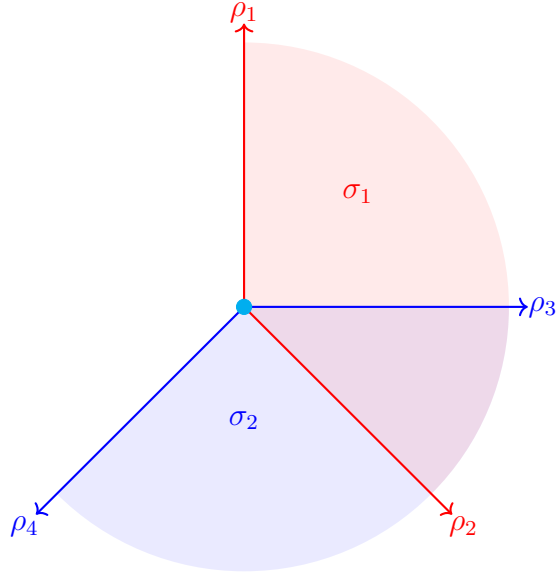

The tangent bundle $\mathscr{T}_X$ is given by the following data:\\ 	$ \mathscr{T}^{[\rho_1,1]}(s) = \left\{ \begin{array}{cc}
	%{r@{\quad \quad}l}
	0 & s \leq -2 \\ 
	\text{Span}(0,1) & s=-1 \\
	k^2 & s \geq 0,
	\end{array} \right.$,
 $\mathscr{T}^{[\rho_2,1]}(s) = \left\{ \begin{array}{cc}
	%{r@{\quad \quad}l}
	0 & s \leq -2 \\ 
	\text{Span}(1,-1) & s=-1 \\
	k^2 & s \geq 0,
	\end{array} \right.$,\\
$\mathscr{T}^{[\rho_3,2]}(s) = \left\{ \begin{array}{cc}
	%{r@{\quad \quad}l}
	0 & s \leq -2 \\ 
	\text{Span}(1,0) & s=-1 \\
	k^2 & s \geq 0,
	\end{array} \right.    
$ and
 $ \mathscr{T}^{[\rho_4,2]}(s) = \left\{ \begin{array}{cc}
	%{r@{\quad \quad}l}
	0 & s \leq -2 \\ 
	\text{Span}(-1,-1) & s=-1 \\
	k^2 & s \geq 0.
	\end{array} \right.
    $\\
By Remark \ref{rem:dist-lattice}, the tangent bundle does not split.

\end{example}

\begin{example}\label{ex:homog-spectra}
    In this example, we consider a toric prevariety occurring as the homogeneous spectra of multigraded polynomial algebras as in \cite{BS_Amplefamilies}. Let $X$ be the blow-up of $\mathbb{P}^2$ at a torus-fixed point $P$. This is, in fact, the first Hirzebruch surface. Then fan of $X$ has four maximal cones    
        \begin{align*}
                &\R_{\geq 0}(1,0) \oplus \R_{\geq 0}(1,1) &\R_{\geq 0}(0,1) \oplus \R_{\geq 0}(1,1) \\
            &\R_{\geq 0}(1,0) \oplus \R_{\geq 0}(-1,-1) &\R_{\geq 0}(0,1) \oplus \R_{\geq 0}(-1,-1)
        \end{align*}
    and their faces. The invariant prime divisors in $X$ are denoted by $D_0, \, D_1, \, D_2$ and $E$ corresponding to the rays $\R_{\geq 0}(-1,-1), \R_{\geq 0}(1,0), \R_{\geq 0}(0,1)$ and $\R_{\geq 0}(1,1)$ respectively.  Then $\text{Pic}(X) \cong \text{Cl}(X)$ is a free $\Z$-module of rank $2$, generated by the divisor classes $[D_0]$ and $[D_0-E]$. The Cox ring of $X$, denoted by $\text{Cox}(X)$, is a polynomial ring in four variables $k[T_0, T_1, T_2, T_3]$, graded by $\text{Pic(X)}$, where $\deg T_i$ is equal to $[D_i]$ for $i=1, 2, 3$ and $\deg T_3$ is equal to $[E]$. Let us consider the map $d: \Z^4 \rar \text{Pic}(X)$ given by $e_i \mapsto \deg T_{i-1}$, where $e_1, \, e_2, \, e_3, \, e_4$ denote the standard basis of $\Z^4$. Then the kernel of $d$ is denoted by $M$, which is a free $\Z$-module of rank $2$ generated by $e_2 + e_4 - e_1, \, e_3 + e_4 - e_1$. Then by \cite[Proposition 3.4]{BS_Amplefamilies}, the homogenous spectra of the $\text{Pic}(X)$-graded polynomial ring $\text{Cox}(X)$, denoted by $\text{Proj}^{\text{Pic}(X)}(\text{Cox}(X))$, is a simplicial toric prevariety with the action of the torus $\text{Spec} (k[M])$. This is non-separated, as discussed in \cite[Example 4.12]{KU}. We describe the associated affine system of fans $\calS$ in the vector space $\R^2 \cong N \otimes_{\Z} \R$ associated to the co-character lattice $N= \text{Hom} (M, \Z)$, following \cite[Remark 3.7]{BS_Amplefamilies}. Let $I=\{1, 2, \ldots, 5\}$. The maximal $T$-invariant affine open subsets of the toric prevariety correspond to the following irreducible fans:
    \begin{equation*}
        \begin{split}
             &\Delta_{11} \text{ consists of the faces of } \R_{\geq 0}(1,0) \oplus \R_{\geq 0}(1,1),\\
             &\Delta_{22} \text{ consists of the faces of }\R_{\geq 0}(0,1) \oplus \R_{\geq 0}(1,1),\\
             &\Delta_{33} \text{ consists of the faces of }\R_{\geq 0}(1,0) \oplus \R_{\geq 0}(0,1),\\
             &\Delta_{44} \text{ consists of the faces of }\R_{\geq 0}(1,0) \oplus \R_{\geq 0}(-1,-1),\\
             &\Delta_{55} \text{ consists of the faces of }\R_{\geq 0}(0,1) \oplus \R_{\geq 0}(-1,-1).\end{split}
    \end{equation*}
    The lower-dimensional fans are given by:
      \begin{equation*}
        \begin{split}
    &\Delta_{12}=\{\R_{\geq 0}(1,1), (0,0)\},\hspace{60pt} \Delta_{13}=\{\R_{\geq 0}(1,0), (0,0)\},\\
             &\Delta_{14}=\{\R_{\geq 0}(1,0), (0,0)\}, \hspace{60pt} \Delta_{15}=\{ (0,0)\}, \\ &\Delta_{23}=\{\R_{\geq 0}(0,1), (0,0)\}, \hspace{60pt} \Delta_{24}=\{ (0,0)\}, \\
             &\Delta_{25}=\{\R_{\geq 0}(0,1), (0,0)\}, \hspace{60pt} \Delta_{34}=\{\R_{\geq 0}(1,0), (0,0)\}, \\
             &\Delta_{35}=\{\R_{\geq 0}(0,1), (0,0)\}, \hspace{60pt} \Delta_{45}=\{\R_{\geq 0}(-1,-1), (0,0)\}.
                \end{split}
    \end{equation*}

\begin{figure}[ht]
    \centering
   \begin{tikzpicture}[scale=0.5]
\fill[blue!40!white, opacity=0.2] (0,0)--(5,0) arc (0: 45:5);

\fill[yellow!40!white] (0,0) --(0,5) arc(90:45:5); 

\fill[red!40!white, opacity=0.2] (0,0) -- (7,0) arc (0:90:7);

\fill[green!30!white, opacity=0.2] (0,0)-- (7,0) arc (0:-135:7);

\fill[teal!30!white, opacity=0.2] (0,0)--(0,7) arc (90:225:7);
     \draw[brown, thick, ->] (0,0) -- (7.5,0);

     \draw[violet, thick, ->] (0,0) -- (0,7.5);

     \draw[olive!70!black, thick, ->] (0,0) -- (6,6);

     \draw[pink!70!black, thick, ->] (0,0) -- (-6, -6);

     \filldraw[red] (0,0) circle (0.2cm);
\node at (3,1) {\textcolor{blue}{$\Delta_{11}$}};

\node at (1,3) {\textcolor{yellow!50!black}{$\Delta_{22}$}};

\node at (3,5) {\textcolor{purple}{$\Delta_{33}$}};

\node at (1,-3) {\textcolor{green!70!black}{$\Delta_{44}$}};

\node at (-3,2) {\textcolor{teal}{$\Delta_{55}$}};

\node at (7.9,0) {\textcolor{brown}{$\rho_1$}};

\node at (7,7) {\textcolor{olive!70!black}{$\rho_2$}};

\node at (0,7.8) {\textcolor{violet}{$\rho_3$}};

\node at (-7,-7) {\textcolor{pink!70!black}{$\rho_4$}};
     
 \end{tikzpicture}
    \caption{Affine system of fans describing the toric prevariety in Example \ref{ex:homog-spectra}}
    \label{fig:homog-spectra}
\end{figure}
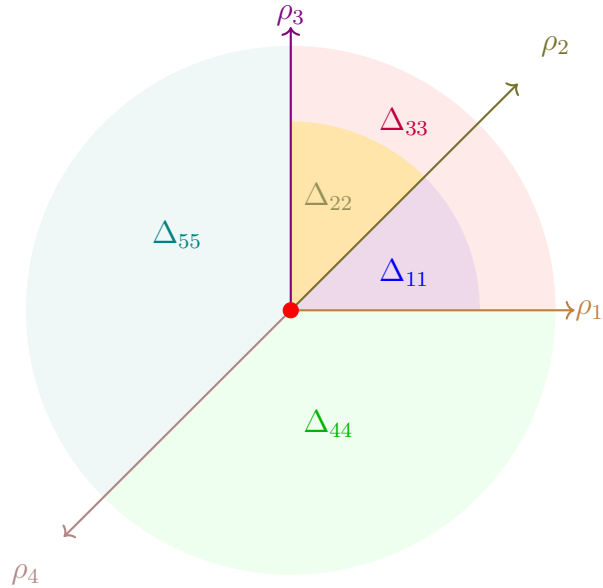

\noindent    
The indexing set is given by $\Lambda = \{[\rho_1,1], \, [\rho_2, 
 1], \, [\rho_3, 2], \, [\rho_4, 4] \}$, where $\rho_1=\R_{\geq 0}(1,0), \, \rho_2= \R_{\geq 0}(1,1), \, \rho_3= \R_{\geq 0}(0,1)$ and $\rho_4=\R_{\geq 0}(-1,-1)$ (see Figure \ref{fig:homog-spectra}). 
The tangent bundle is given by the following data:\\ 	$ \mathscr{T}^{[\rho_1,1]}(s) = \left\{ \begin{array}{cc}
	%{r@{\quad \quad}l}
	0 & s \leq -2 \\ 
	\text{Span}(1,0) & s=-1 \\
	k^2 & s \geq 0,
	\end{array} \right.$,
 $\mathscr{T}^{[\rho_2,1]}(s) = \left\{ \begin{array}{cc}
	%{r@{\quad \quad}l}
	0 & s \leq -2 \\ 
	\text{Span}(1,1) & s=-1 \\
	k^2 & s \geq 0,
	\end{array} \right.$,\\
$\mathscr{T}^{[\rho_3,2]}(s) = \left\{ \begin{array}{cc}
	%{r@{\quad \quad}l}
	0 & s \leq -2 \\ 
	\text{Span}(0,1) & s=-1 \\
	k^2 & s \geq 0,
	\end{array} \right.    
$ and
 $ \mathscr{T}^{[\rho_4, 2]}(s) = \left\{ \begin{array}{cc}
	%{r@{\quad \quad}l}
	0 & s \leq -2 \\ 
	\text{Span}(-1,-1) & s=-1 \\
	k^2 & s \geq 0.
	\end{array} \right.
    $\\
Here again, following Remark \ref{rem:dist-lattice}, the tangent bundle does not split.
\end{example}

\noindent{\bf Acknowledgements.} We would like to thank the referee for their valuable comments and suggestions that
have led to several improvements in both content and presentation.

\noindent
\textbf{Funding }The first author is supported by the Indian Institute of Technology (Indian School of Mines) Faculty Research Scheme grant FRS(219)/2024-2025/M\&C.
The second author is supported by the Israel Science Foundation, grant No. 1405/22.

\bibliographystyle{alpha}
\bibliography{prevar}

\end{document}